\documentclass[12pt]{amsart}

\usepackage{amssymb, amscd, txfonts}
\usepackage[all]{xy}
\usepackage{graphicx}

\numberwithin{equation}{section}

\newtheorem{thm}{Theorem}[section]
\newtheorem{prop}[thm]{Proposition}
\newtheorem{lem}[thm]{Lemma}
\newtheorem{cor}[thm]{Corollary}
\newtheorem{q}[thm]{Question}

\theoremstyle{definition}
\newtheorem{defn}[thm]{Definition}

\theoremstyle{remark}
\newtheorem{rem}[thm]{Remark}

\renewcommand{\hom}{\operatorname{Hom}}
\renewcommand{\ker}{\operatorname{Ker}}

\newcommand{\Z}{\mathbb{Z}}

\newcommand{\R}{\mathbb{R}}
\newcommand{\C}{\mathbb{C}}
\newcommand{\F}{\mathbb{F}}

\DeclareMathOperator{\Ad}{Ad}

\DeclareMathOperator{\coker}{Coker}
\DeclareMathOperator{\im}{Im}

\begin{document}

%%%%%%% Title %%%%%%%%%%%%%%%%%%%%%%%%%%%%%%%%%%%%%%%%%%%%%%%%%%%%%%%%%%%%%%%%%
\title[Torsion functions on moduli spaces]
{Torsion functions on moduli spaces in view of the cluster algebra}
\author[T.~Kitayama]{Takahiro KITAYAMA}
\author[Y.~Terashima]{Yuji Terashima}
\address{Department of Mathematics, Tokyo Institute of Technology,
2-12-1 Ookayama, Meguro-ku, Tokyo 152-8551, Japan}
\email{kitayama@math.titech.ac.jp}
\address{Department of Mathematical and Computing Sciences, Tokyo Institute of
Technology, 2-12-1 Ookayama, Meguro-ku, Tokyo 152-8550, Japan}
\email{tera@is.titech.ac.jp}

\begin{abstract}
We introduce non-acyclic $PGL_n(\C)$-torsion of a $3$-manifold with toroidal
boundary as an extension of J.~Porti's $PGL_2(\C)$-torsion, and present an
explicit formula of the $PGL_n(\C)$-torsion of a mapping torus for a surface
with punctures, by using the higher Teichm\"{u}ler theory due to V.~Fock and
A.~Goncharov.
Our formula gives a concrete rational function which represents the torsion
function and comes from a concrete cluster transformation associated with the
mapping class.
\end{abstract}

\maketitle

%%%%%%% Section 1 %%%%%%%%%%%%%%%%%%%%%%%%%%%%%%%%%%%%%%%%%%%%%%%%%%%%%%%%%%%%%
\section{Introduction}
In the important work \cite{P} J.~Porti introduced non-acyclic
$PGL_2(\C)$-torsion of a $3$-manifold with toroidal boundary, and began to
study the torsion as a function on the moduli space of
$PGL_2(\C)$-representations of the fundamental group. 
In particular, in the case of a mapping torus for the once-punctured torus, he
gave a concrete way to compute the torsion function, by using trace functions.

In this paper we introduce non-acyclic $PGL_n(\C)$-torsion of a $3$-manifold
with toroidal boundary, and present an explicit formula of the
$PGL_n(\C)$-torsion of a mapping torus for a general surface with punctures,
by using the higher Teichm\"{u}ler theory due to V.~Fock and
A.~Goncharov~\cite{FG}.
See Theorems 4.1 and 4.2 for the precise statement of our main theorems.
Our formulas, with methods developed in~\cite{TY, NTY}, give concrete rational
functions which represent the functions induced by twisted Alexander
polynomials and the non-acyclic torsion on components of the
$PGL_n(\C)$-character variety.
The rational functions come from a concrete cluster
transformation~\cite{FZ1,FZ2} associated with the mapping class. 
Moreover, we show that for any pseudo-Anosov mapping class of a surface, the
conjugacy class of a holonomy representation of the mapping torus is contained
in the components.

Other attempts to define non-acyclic $PGL_n(\C)$-torsion and to give formulas in terms of quantities closely related to cluster variables should be remarked.
In \cite{MFP3} P.~Menal-Ferrer and J.~Porti defines non-acyclic $PGL_n(\C)$-torsion of a $3$-manifold by another method, and shows an explicit relationship between its asymptotic behavior on $n$ and the volume of the manifold, extending the result of M\"uller for closed manifolds \cite{Mu}.
In \cite{DG} T.~Dimofte and S.~Garoufalidis defines a series of invariants in terms of the shapes together with the gluing equations of an ideal triangulation of a $3$-manifold, and conjectures that each invariant of the series agree with each term of the asymptotic expansion of the Kashaev invariant of the manifold.
In particular, its first one of the series should conjecturelly give non-acyclic $PGL_2(\C)$-torsion, and they verify this experimentally for a large class of $3$-manifolds.
In \cite{GGZ, GTZ} S.~Garoufalidis, M.~Goerner, D.~P.~Thurston and
C.~K.~Zickert study moduli spaces of higher dimensional representations for a general $3$-manifold
in terms of analogous coordinates to Fock and Goncharov's associated to an ideal triangulation of the manifold itself.
It is interesting to obtain an explicit formula of the $PGL_n(\C)$-torsion for
a general $3$-manifold, with a combination of the above results and our method.

This paper is organized as follows.
In Section $2$, following Fock and Goncharov \cite{FG}, we review cluster
algebras associated to an ideal triangulation of a punctured surface and then
show that the characters of geometric representations of mapping tori are described by the
cluster variables.
Section $3$ is devoted to introduce and study non-acyclic Reidemeister torsion
for higher dimensional representations.
In Section $4$ we prove the main theorems, and demonstrate our
theory with concrete examples.

\subsection*{Acknowledgment}
The authors would like to thank H.~Fuji, K.~Nagao, Y.~Yamaguchi and M.~Yamazaki
for valuable conversations.
The authors also wishes to express their thanks to the anonymous referee
for several useful comments in revising the manuscript.

%%%%%%% Section 2 %%%%%%%%%%%%%%%%%%%%%%%%%%%%%%%%%%%%%%%%%%%%%%%%%%%%%%%%%%%%%
\section{Character varieties and cluster algebras}

\subsection{Character varieties} \label{subsec_CV}
We begin with reviewing some of the standard facts on character varieties.
See Lubotzky and Magid~\cite{LM} for more details.

Let $S$ be a compact connected oriented surface with $m$ boundary circles.
The group $PGL_n(\C)$ acts on the affine algebraic set
$\hom(\pi_1 S, PGL_n(\C))$ by conjugation.
We denote by $X_{S, n}$ the algebro-geometric quotient of the action, which is
called the \textit{$PGL_n(\C)$-character variety} of $\pi_1 S$.
For a representation $\rho \colon \pi_1 S \to PGL_n(\C)$ we write $\chi_\rho$
for its image by the quotient map and call it the \textit{character} of $\rho$.
We fix representatives $\tilde{\gamma}_1, \dots, \tilde{\gamma}_m \in \pi_1 S$
of the boundary circles of $S$. 
A \textit{framed representation} is a pair of a representation
$\rho \colon \pi_1 S \to PGL_n(\C)$ and Borel subgroups $B_1, \dots, B_m$ of
$PGL_n(\C)$ such that $\rho(\tilde{\gamma}_i) \in B_i$ for all $i$.
The set $\widetilde{\mathcal{X}}_{S, n}$ of framed representations is a closed
subset of the affine algebraic set
$\hom(\pi_1 S, PGL_n(\C)) \times \mathcal{B}^m$, where $\mathcal{B}$ is the
flag variety of $PGL_n(\C)$ parameterizing Borel subgroups.
The $PGL_n(\C)$ acts on $\widetilde{\mathcal{X}}_{S, n}$ by conjugation.
We denote by $\mathcal{X}_{S, n}$ the algebro-geometric quotient of the action,
and for a framed representation $(\rho, B_1, \dots, B_m)$ we write
$\chi_{(\rho, B_1, \dots, B_m)}$ for its image by the quotient map.
Forgetting framings $(B_1, \dots, B_m)$ gives a regular map
$\widetilde{\mathcal{X}}_{S, n} \to \hom(\pi_1 S, PGL_n(\C))$.
We denoted by $\pi \colon \mathcal{X}_{S, n} \to X_{S, n}$ the induced map on
the quotients.

The tangent space $T_{\chi_\rho} X_{S, n}$ is identified with a subspace of the
$1$st twisted group cohomology
$H_{\Ad \circ \rho}^1(\pi_1 S; \mathfrak{pgl}_n(\C))$ by the monomorphism given by
\[ \left. \frac{d \chi_{\rho_t}}{dt} \right|_{t=0} \mapsto
\left[ \gamma \mapsto \left. \frac{d \rho_t(\gamma) \rho_t(\gamma^{-1})}{dt} \right|_{t=0} \right],
\]
where $\rho_0 = \rho$ and $\gamma \in \pi_1 S$ \cite{W}.
It is easily seen that the map
$T_{(\rho, B_1, \dots, B_m)} \widetilde{\mathcal{X}}_{S, n} \to T_\rho \hom(\pi_1, PGL_n(\C))$
is an epimorphism, and so is
$(d \pi)_{\chi_{(\rho, B_1, \dots, B_m)}} \colon T_{\chi_{(\rho, B_1, \dots, B_m)}} \mathcal{X}_{S, n} \to T_{\chi_\rho} X_{S, n}$.

We denote by $\Gamma_S$ the mapping class group of $S$ which is defined to be
the group of isotopy classes of orientation preserving homeomorphisms of $S$,
where these isotopies are understood to fix $\partial S$ pointwise.
For $\varphi \in \Gamma_S$ we write $M_\varphi$ for the mapping torus
$S \times [0, 1] / (x, 1) \sim (\varphi(x), 0)$ of $\varphi$.
A mapping class $\varphi \in \Gamma_S$ induces automorphisms $\varphi^*$ on
$\mathcal{X}_{S, n}$ and $X_{S, n}$ by pullback of representations.
For a representation
$\rho \colon \pi_1 M_\varphi \to PGL_n(\C)$, $\chi_{\rho|_{\pi_1 S}}$ is
contained in the fixed point set $X_{S, n}^{\varphi^*}$ of
$\varphi^* \colon X_{S, n} \to X_{S, n}$.

\subsection{Cluster algebras associated to an ideal triangulation}

We review cluster algebras for $S$, following \cite{FG}.
Here, in particular, we only consider \textit{$y$-variables}.
See \cite{FZ1, FZ2} for more details on cluster algebras.
In the following we assume that $\partial S$ is non-empty and that if the genus
of $S$ is $0$, then the number $m$ of the boundary circles is greater than $3$.

%We recall the cluster algebra associated to a quiver in a special form for our purpose.

Let $Q$ be a quiver with the vertex set $I = \{ 1, 2, \dots, l \}$ and without
loops and oriented $2$-cycles.
For $i, j \in I$ we set
\[ \epsilon_{ij} := \sharp \{ \text{oriented edges from $i$ to $j$} \}
-\sharp \{ \text{oriented edges from $j$ to $i$} \}. \]
Note that $Q$ is uniquely determined by the skew-symmetric matrix
$\epsilon_{ij}$.
For $k \in I$ the \textit{mutation} $\mu_k Q$ at $k \in I$ is defined by the
following matrix $\epsilon_{ij}'$:
\[ \epsilon_{ij}' =
\begin{cases}
-\epsilon_{ij} &\text{if $k \in \{ i, j \}$}, \\
\epsilon_{ij}
+\frac{|\epsilon_{ik}| \epsilon_{kj} + \epsilon_{ik} |\epsilon_{kj}|}{2}
&\text{if $k \notin \{ i, j \}$}.
\end{cases}
\]
A complex torus
\[ \mathcal{X}_Q := (\C^*)^I \]
is associated to $Q$.
Let $(y_1, \dots, y_l)$ be the standard coordinates on the torus.
For $k \in I$ a rational map
$(\mu_k)_* \colon \mathcal{X}_Q \to \mathcal{X}_{\mu_k Q}$ associated to the
mutation $\mu_k Q$ is defined by the following:
\begin{align*}
\text{the $i$th coordinate of } (\mu_k)_*(y_1, \dots, y_l) &=
\begin{cases}
y_i^{-1} & \text{if $i=k$}, \\
y_i (1+y_k^{-1})^{-\epsilon_{ik}}
&\text{if $i \neq k$ and $\epsilon_{ik} \geq 0$}, \\
y_i (1+y_k)^{-\epsilon_{ik}} & \text{if $i \neq k$ and $\epsilon_{ik} \leq 0$}.
\end{cases}
\end{align*}

Shrinking each component of $\partial S$, we get a closed surface
$\overline{S}$ with marked points.
A triangulation of $\overline{S}$ with vertices at the marked points is called
an \textit{ideal triangulation} of $S$.
In this paper we only consider an ideal triangulation without self-folded
edges.
Such a triangulation exists under the above assumption on $S$.  

Let $T$ be an ideal triangulation of $S$ and let $n$ be a positive integer.
We identify each triangle of $T$ with the triangle
\[ x+y+z=n, \quad x, y, z > 0 \]
and consider its triangulation given by the lines $x=p$, $y=p$, $z=p$ where
$0 \leq p \leq n$ is an integer.
The subtriangulation $T_n$ of $T$ is called the \textit{$n$-triangulation} of
$T$. 
The quiver $Q_{T, n}$ associated to $T_n$ is defined as:
\[ Q_{T, n} := \overline{T_n^{(1)} \setminus T^{(1)}}, \]
where $T^{(1)}$ and $T_n^{(1)}$ are the $1$-skeletons of $T$ and $T_n$
respectively.
(See Figure \ref{fig_quiver}.)
The vertex set $I_{T, n}$ of $Q_{T, n}$ consists of vertices of $T_n$ except
the marked points of $\overline{S}$.
The orientation of each edge of $Q_{T, n}$ is provided by that of $S$ as
follows.
Take a triangle $\Delta$ of $T$, which is oriented as a subspace of $S$.
Then each edge of $Q_{T, n}$ contained in $\Delta$ is oriented so that the
direction is parallel to one of the boundary edge of $\Delta$. 
For simplicity of notation, we set
$\mathcal{X}_{T, n} := \mathcal{X}_{Q_{T, n}}$.
Writing $e$ and $f$ for the number of edges and faces of $T$ respectively, we have 
\begin{align*}
|I_{T, n}| &= (n-1)e + \frac{(n-1)(n-2)}{2} f, \\
\chi(S) &= -e+f, \\
2e &= 3f.
\end{align*}
These imply the formula
\[ \dim \mathcal{X}_{T, n} = |I_{T, n}| = - (n^2 - 1) \chi(S). \]
In the following we set
\[ l = - (n^2 - 1) \chi(S). \]

\begin{figure}[h]
\centering
\includegraphics[width=12cm, clip]{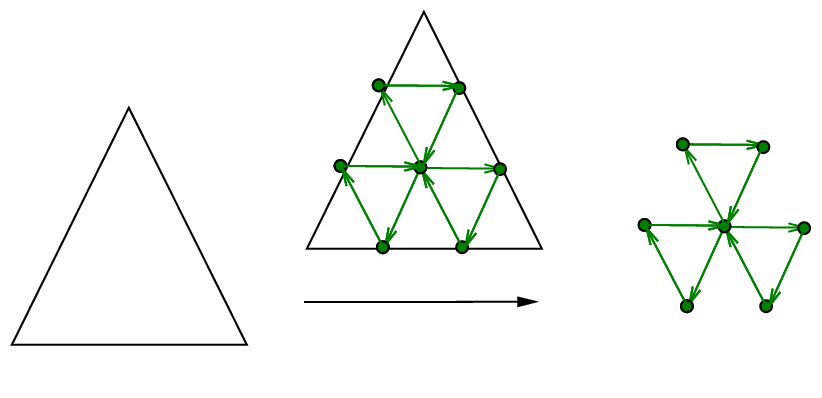}
\caption{The $3$-triangulation $T_3$ and the quiver $Q_{T, 3}$}
\label{fig_quiver}
\end{figure}

Fock and Goncharov~\cite[Section 9]{FG} constructed a regular map
$\nu_T \colon \mathcal{X}_{T, n} \to \mathcal{X}_{S, n}$ and a rational map
$\varphi_T \colon \mathcal{X}_{S, n} \to \mathcal{X}_{T, n}$ such that
$\varphi_T \circ \nu_T = id$.
In particular, $\nu_T$ is an embedding, and \cite[Theorem 9.1]{FG} implies that
the images for all the triangulations cover $\mathcal{X}_{S, n}$.
The regular map $\nu_T \colon \mathcal{X}_{T, n} \to \mathcal{X}_{S, n}$ is explicitly constructed in \cite[Section 9.10]{FG}, and the rational map $\varphi_T \colon \mathcal{X}_{S, n} \to \mathcal{X}_{T, n}$ is in \cite[Section 9.3]{FG}.
For $(y_1, \dots, y_l) \in \mathcal{X}_{T, n}$ we set
\[ \chi_{(y_1, \dots, y_l)} := \pi \circ \nu_T((y_1, \dots, y_l)). \]

\begin{rem}
Fock and Goncharov associated cluster algebras also to an ideal triangulation with self-folded edges.
See \cite[Section 10.7]{FG} for the treatment of the case.
\end{rem}

\subsection{The mapping class group actions on cluster variables}

Here we define the action $\varphi^* \colon \mathcal{X}_{T, n} \to \mathcal{X}_{T, n}$ for each $\varphi \in \Gamma_S$.
The definition plays important role to relate the moduli space for $M_\varphi$ to cluster algebras for the fiber surface.
In fact the fixed point set $\mathcal{X}_{T, n}^{\varphi^*}$ parameterizing $\mathcal{X}_{S, n}^{\varphi^*}$ makes sense.

A mapping class $\varphi \in \Gamma_S$ naturally induces a bijection
$\varphi_* \colon I_{T, n} \to I_{\varphi(T), n}$.
It defines an isomorphism
$\sigma \colon \mathcal{X}_{\varphi(T), n} \to \mathcal{X}_{T, n}$ by
\[ \sigma(y_{\varphi_*(1)}, \dots, y_{\varphi_*(l)}) = (y_1, \dots, y_l). \]
The isomorphism $\sigma$ is called the \textit{labeling change} of $\varphi$.
This is essential for obtaining the \textit{genuine} action $\varphi^* \colon \mathcal{X}_{T, n} \to \mathcal{X}_{T, n}$ defined later.

\begin{prop} \label{prop_C}
For $\varphi \in \Gamma_S$ the following diagram commutes: 
\[
\begin{CD}
\mathcal{X}_{\varphi(T), n} @>\sigma>> \mathcal{X}_{T, n} \\
@V \nu_{\varphi(T)} VV @V \nu_T VV \\
\mathcal{X}_{S, n} @>\varphi^*>> \mathcal{X}_{S, n}.
\end{CD}
\]
\end{prop}

\begin{proof}
We first briefly overview the flow of the construction of the map $\nu_T \colon \mathcal{X}_{T, n} \to \mathcal{X}_{S, n}$.
Let $\Gamma$ be the $1$-skeleton of a dual complex of $T$.
Replacing edges and vertices of $\Gamma$ by rectangles and hexagons respectively, we obtain a decomposition of $S$.
The orientation of $S$ naturally induces that of each edge of the decomposition.
We denote by $\Delta$ the set of oriented edges of the decomposition.  
It follows from \cite[Lemma 9.6]{FG} that $\mathcal{X}_{S, n}$ can be regarded as a quotient of $PGL_n(\C)^\Delta$.
For $(y_1, \dots, y_l) \in \mathcal{X}_{T, n}$ a representative of $\nu_T(y_1, \dots, y_l)$ in $PGL_n(\C)^\Delta$ is explicitly given as in \cite[Theorem 9.2]{FG}.

Let $e \in \Delta$. 
If $e$ is an edge of a rectangle, then let $v_1, \dots, v_q$ be the vertices of $Q_{T, n}$ on the two triangles sharing the edge corresponding with the rectangle.
If $e$ is an edge of a hexagon, then let $v_1, \dots, v_q$ be the vertices of $Q_{T, n}$ on the triangle corresponding with the hexagon.
It follows from the construction~\cite[Theorem 9.2]{FG} of
$\nu_T \colon \mathcal{X}_{T, n} \to \mathcal{X}_{S, n}$ that for
$(y_1, \dots, y_l) \in \mathcal{X}_{T, n}$, $\nu_T(y_1, \dots, y_l)$ is presented by an element of
$PGL_n(\C)^\Delta$ whose image of $e$ is determined only by the coordinates
$(y_{v_1}, \dots, y_{v_q})$ in $(y_1, \dots, y_l)$, and that $\nu_{\varphi(T)}(\sigma^{-1}(y_1, \dots, y_l))$
is represented by one whose image of $\varphi(e)$ is similarly determined by
$(y_{\varphi(v_1)}, \dots, y_{\varphi(v_q)})$ in $\sigma^{-1}(y_1, \dots, y_l)$, which are
equal to $(y_{v_1}, \dots, y_{v_q})$ in $(y_1, \dots, y_l)$.
Therefore
\[ \varphi^* \circ \nu_{\varphi(T)} \circ \sigma^{-1}(y_1, \dots, y_l) = \nu_T(y_1, \dots, y_l) \]
for any $(y_1, \dots, y_l) \in \mathcal{X}_{T, n}$, and the lemma follows.
\end{proof}

Let $T'$ be an ideal triangulation of $S$ obtained from $T$ by a flip $f$ at an
edge $e$.
We identify each of two triangles sharing $e$ as a face with the triangle
\[ x+y+z=n, \quad x, y, z > 0 \]
so that the edge on the line $x=0$ represents $e$.
Let $v^0_1, \dots, v^0_{n-2}$ be the vertices of $Q_{T, n}$ on the line
$x = 0$, and let $v^i_1, \dots, v^i_{n-i-2}$ and $w^i_1, \dots, w^i_{n-i-2}$ be
these on the line $x=i$ contained in the interior of the two triangles for
$1 \leq i \leq n-2$.
Then the following composition of mutations change $Q_{T, n}$ into $Q_{T', n}$
~\cite[Proposition 10.1]{FG}:
\[ \mu^{n-2} \circ \dots \circ \mu^0, \]
where
\begin{align*}
\mu^0 &:= \mu_{v^0_1} \circ \dots \circ \mu_{v^0_{n-1}}, \\
\mu^i &:= (\mu_{v^i_1} \circ \dots \circ \mu_{v^i_{n-i-1}}) \circ (\mu_{w^i_1} \circ \dots \circ \mu_{w^i_{n-i-1}}), \quad \text{for $1 \leq i \leq n-2$}.
\end{align*}
A rational map $f_* \colon \mathcal{X}_{T, n} \to \mathcal{X}_{T', n}$ is
defined as
\[ f_* := (\mu^{n-2})_* \circ \dots \circ (\mu^0)_*. \]
The following commutative diagram is proved in \cite[Sections 10.5 and 10.6]{FG}:
\[ \xymatrix{
\mathcal{X}_{T, n} \ar[rr]^{f_*} \ar[dr]_{\nu_T} & & \mathcal{X}_{T', n} \ar[dl]^{\nu_{T'}} \\
& \mathcal{X}_{S, n} & 
} \]

\begin{defn}
We take a sequence $f_1, \dots, f_q$ of flips changing $T$ into $\varphi(T)$
and define a rational map
$\varphi \colon \mathcal{X}_{T, n} \to \mathcal{X}_{T, n}$ as
\[ \varphi^* = \sigma \circ (f_q)_* \circ \dots \circ (f_1)_*, \]
where $\sigma$ is the labeling change of $\varphi$.
\end{defn}

The following is now a direct consequence of Proposition \ref{prop_C}.

\begin{cor} \label{cor_equivariance}
For $\varphi \in \Gamma_S$ the following diagram commutes: 
\[
\begin{CD}
\mathcal{X}_{T, n} @>\varphi^*>> \mathcal{X}_{T, n} \\
@V \nu_T VV @V \nu_T VV \\
\mathcal{X}_{S, n} @>\varphi^*>> \mathcal{X}_{S, n}.
\end{CD}
\]
\end{cor}

Note that it follows from the above corollary that
$\varphi^* \colon \mathcal{X}_{T, n} \to \mathcal{X}_{T, n}$ does not depend on
the choice of a sequence of flips.

\subsection{The character of a holonomy representation}

We show that the characters of geometric representations of mapping tori are
described by cluster variables.

It is well-known that for $\varphi \in \Gamma_S$ the mapping torus $M_\varphi$
has a hyperbolic structure if and only if $\varphi$ is pseudo-Anosov~\cite{Th}.

\begin{thm} \label{thm_holonomy}
Let $\varphi \in \Gamma_S$ be pseudo-Anosov and
$\rho \colon \pi_1 M_\varphi \to PGL_2(\C)$ a holonomy representation of
$M_\varphi$.
Then there exists $y_i \in \C^*$ for $i = 1, \dots, l$ such that
$\chi_{\rho|_{\pi_1 S}} = \chi_{(y_1, \dots, y_l)}$.
\end{thm}

\begin{proof}
Since for any representative $\tilde{\gamma} \in \pi_1 S$ of a boundary circle
of $S$ a Borel subgroup containing $\rho(\tilde{\gamma})$ is uniquely
determined, $\pi^{-1}(\chi_{\rho|_{\pi_1 S}})$ consists of one point
$\chi_{(\rho|_{\pi_1 S}, B_1, \dots, B_m)}$.
It suffices to show that the rational map
$\varphi_T \colon \mathcal{X}_{S, 2} \to \mathcal{X}_{T, 2}$ is defined on the
point, since, if so, then
$\varphi_T(\chi_{(\rho|_{\pi_1 S}, B_1, \dots, B_m)}) \in \mathcal{X}_{T, 2}$
satisfies the desired condition.

Let $e$ be an edge of $T$ and let $\Gamma$ be the $1$-skeleton of a dual
complex of $T$.
Write $x, y, z, t$ for the vertices of two triangles of $T$ sharing $e$ so that
$xtz$ and $xzy$ are the triangles compatible with the orientations coming from
that of $S$.
There are natural $4$ (unoriented) loops
$\gamma_x, \gamma_y, \gamma_z, \gamma_t$ in $\Gamma$ starting at a point on the
dual edge of $e$ and going around the boundary of the dual cells of the
vertices $x, y, z, t$ respectively.
Take representatives
$\tilde{\gamma}_x, \tilde{\gamma}_y, \tilde{\gamma}_z, \tilde{\gamma}_t \in \pi_1 S$
of $\gamma_x, \gamma_y, \gamma_z, \gamma_t$ with any orientations respectively, and let
$\lambda_x, \lambda_y, \lambda_z, \lambda_t \in \C P^1$ be the fixed point of
the M\"{o}bius transformations
$\rho(\tilde{\gamma}_x), \rho(\tilde{\gamma}_y), \rho(\tilde{\gamma}_z), \rho(\tilde{\gamma}_t)$
respectively.
Then, if defined, the coordinate $y_e$ of
$\varphi_T(\chi_{(\rho|_{\pi_1 S}, B_1, \dots, B_m)})$ corresponding to the
vertex on $e$ is given by
\[ y_e = \frac{(\lambda_x - \lambda_t)(\lambda_y - \lambda_z)}{(\lambda_z - \lambda_t)(\lambda_x - \lambda_y)}. \]
See \cite[Sections 9.3 and 9.5]{FG} for the definition of the coordinate
functions.
Since $\tilde{\gamma}_x, \tilde{\gamma}_y, \tilde{\gamma}_z, \tilde{\gamma}_t$
are distinct nontrivial elements of the free group $\pi_1 S$ and since
$\rho \colon \pi_1 M_\varphi \to PGL_n(\C)$ is faithful,
$\rho(\tilde{\gamma}_x), \rho(\tilde{\gamma}_y), \rho(\tilde{\gamma}_z), \rho(\tilde{\gamma}_t)$
are non-commutative with each other, and so
$\lambda_x, \lambda_y, \lambda_z, \lambda_t$ are all distinct elements. 
Therefore the value $y_e$ is nonzero for each $e$, which implies that
$\varphi_T \colon \mathcal{X}_{S, 2} \to \mathcal{X}_{T, 2}$ is defined on
$\chi_{(\rho|_{\pi_1 S}, B_1, \dots, B_m)}$.
\end{proof}

\begin{cor} \label{cor_solution}
For any pseudo-Anosov mapping class $\varphi \in \Gamma_S$ the fixed point set
$\mathcal{X}_{T, n}^{\varphi^*}$ is nonempty.
\end{cor}

\begin{proof}
Let $\iota_n \colon \mathcal{X}_{T, 2} \to \mathcal{X}_{T, n}$ be the map
defined as follows.
For $(y_1, \dots, y_l) \in \mathcal{X}_{T, 2}$, each coordinate of
$\iota_n(y_1, \dots, y_l)$ corresponding to a vertex of $Q_{T, n}$ on an edge
of $T$ is defined to be $y_i$ corresponding to the unique vertex of $Q_{T, 2}$
on the same edge, and the other coordinates are all defined to be $1$.
The commutativity of the following diagram is straightforward by the definition
of $\varphi^*$:
\[
\begin{CD}
\mathcal{X}_{T, 2} @>\varphi^*>> \mathcal{X}_{T, 2} \\
@V \iota_n VV @V \iota_n VV \\
\mathcal{X}_{T, n} @>\varphi^*>> \mathcal{X}_{T, n}.
\end{CD}
\]
It follows from this commutativity, Corollary \ref{cor_equivariance} and
Theorem \ref{thm_holonomy} that for $(y_1, \dots, y_l) \in \mathcal{X}_{T, 2}$ in
Theorem \ref{thm_holonomy},
$\iota_n(y_1, \dots, y_l) \in \mathcal{X}_{T, n}^{\varphi^*}$, which proves the
corollary.
\end{proof}

%%%%%%% Section 3 %%%%%%%%%%%%%%%%%%%%%%%%%%%%%%%%%%%%%%%%%%%%%%%%%%%%%%%%%%%%%
\section{Torsion functions}

\subsection{Reidemeister torsion}

First we review basics of Reidemeister torsion.
See Milnor~\cite{M1} and Turaev~\cite{Tu} for more details.

Let $C_* = (C_n \xrightarrow{\partial_n} C_{n-1} \to \cdots \to C_0)$ be a
finite dimensional chain complex over a commutative field $\F$, and let
$c = \{ c_i \}$ and $h = \{ h_i \}$ be bases of $C_*$ and $H_*(C_*)$
respectively.
Choose bases $b_i$ of $\im \partial_{i+1}$ for each $i = 0, 1, \dots n$, and
take a basis $b_i h_i b_{i-1}$ of $C_i$ for each $i$ as follows.
Picking a lift of $h_i$ in $\ker \partial_i$ and combining it with $b_i$, we
first obtain a basis $b_i h_i$ of $C_i$.
Then picking a lift of $b_{i-1}$ in $C_i$ and combining it with $b_i h_i$, we
obtain a basis $b_i h_i b_{i-1}$ of $C_i$.
The \textit{algebraic torsion} $\tau(C_*, c, h)$ is defined as:
\[ \tau(C_*, c, h) := \prod_{i=0}^n
[b_i h_i b_{i-1} / c_i]^{(-1)^{i+1}} ~\in \F^\times, \]
where $[d' / d]$ is the determinant of the base change matrix
from $d$ to $d'$ for bases $d$ and $d'$.
If $C_*$ is acyclic, then we just write $\tau(C_*, c)$.
It can be easily checked that $\tau(C_*, c, h)$ does not depend on the choices
of $b_i$ and $b_i h_i b_{i-1}$.

The algebraic torsion $\tau$ has the following multiplicative property.
Let
\[ 0 \to C_*' \to C_* \to C_*'' \to 0 \]
be a short exact sequence of finite dimensional chain complexes over $\F$ and
let $c = \{ c_i \}, c' = \{ c_i' \}, c'' = \{ c_i'' \}$ and
$h = \{ h_i \}, h' = \{ h_i' \}, h'' = \{ h_i'' \}$ be bases of
$C_*, C_*', C_*''$ and $H_*(C_*), H_*(C_*'), H_*(C_*'')$.
Picking a lift of $c_i''$ in $C_i$ and combining it with the image of $c_i'$ in
$C_i$, we obtain a basis $c_i' c_i''$ of $C_i$.
We denote by $\mathcal{H}_*$ the corresponding long exact sequence in homology,
and by $d$ the basis of $\mathcal{H}_*$ obtained by combining $h, h', h''$.

\begin{lem}(\cite[Theorem 3.\ 1]{M1}) \label{lem_M}
If $[c_i' c_i'' / c_i] = 1$ for all $i$, then
\[ \tau(C_*, c, h) = \tau(C_*', c', h') \tau(C_*'', c'', h'')
\tau(\mathcal{H}_*, d). \]
\end{lem} 

In the following when we write $C_*(\widetilde{Y}, \widetilde{Z})$ for a
CW-pair $(Y, Z)$, $\widetilde{Y}$, $\widetilde{Z}$ stand for the universal
cover of $Y$ and the pullback of $Z$ by the universal covering map
$\widetilde{Y} \to Y$ respectively.
For a $n$-dimensional representation $\rho \colon \pi_1 Y \to GL(V)$ over a
commutative field $\F$ we define the twisted homology group and the cohomology
group associated to $\rho$ as follows:
\begin{align*}
H_i^\rho(Y, Z; V) &:=
H_i(C_*(\widetilde{Y}, \widetilde{Z}) \otimes_{\Z[\pi_1 Y]} V), \\
H_\rho^i(Y, Z; V) &:=
H^i(\hom_{\Z[\pi_1 Y]}(C_*(\widetilde{Y}, \widetilde{Z}), V)).
\end{align*}
If $Z$ is empty, then we write $H_i^\rho(Y; V)$ and $H_\rho^i(Y; V)$
respectively.

For a basis $h$ of $H_*^\rho(Y; V)$ the \textit{Reidemeister torsion}
$\tau_\rho(Y; h)$ associated to $\rho$ and $h$ is defined as follows:
We choose a lift $\tilde{e}$ in $\widetilde{Y}$ for each cell
$e \subset Y$.
Then
\[ \tau_\rho(Y; h) :=
\tau(C_*(\widetilde{Y}) \otimes_{\Z[\pi_1 Y]} V,
\langle \tilde{e} \otimes 1 \rangle_e, h) ~\in
\F^\times / (-1)^n \det \rho(\pi_1 Y). \]
If  $H_*^\rho(Y; V) = 0$, then we drop $h$ in the notation $\tau_\rho(Y; h)$.
It can be easily checked that $\tau_\rho(Y; h)$ does not depend on the
choice of $\tilde{e}$ and is invariant under conjugation of representations.
It is known that Reidemeister torsion is a simple homotopy invariant.

Let $M$ be a compact connected orientable $3$-manifold with empty or toroidal
boundary and let $\psi \colon \pi_1 M \to \langle t \rangle$ be a homomorphism.
For a representation $\rho \colon \pi_1 Y \to GL_n(\F)$ satisfying
$H_*^{\psi \otimes \rho}(Y; \F(t)^n) = 0$, where
$\psi \otimes \rho \colon \pi_1 M \to GL_n(\F(t))$ is given by
$\psi \otimes \rho(\gamma) = \psi(\gamma) \rho(\gamma)$ for
$\gamma \in \pi_1 M$, the Reidemeister torsion $\tau_{\psi \otimes \rho}(M)$ is
known by Kirk and Livingston~\cite{KL}, and Kitano~\cite{K} to be essentially
equal to the \textit{twisted Alexander polynomial} associated to $\psi$ and
$\rho$.
For twisted Alexander polynomials we refer the reader to \cite{FV}.

\subsection{Non-acyclic Reidemeister torsion for higher dimensional representations}

We introduce non-acyclic Reidemeister torsion of a $3$-manifold for higher
dimensional representations as a natural generalization of Porti's torsion for
a $2$-dimensional representation~\cite{P}.

For a compact orientable manifold $Y$ and a representation
$\rho \colon \pi_1 Y \to PGL_n(\C)$ the Killing form of $\mathfrak{pgl}_n(\C)$
induces a non-degenerate intersection pairing:
\begin{equation} \label{eq_duality}
H_i^{\Ad \circ \rho}(Y; \mathfrak{pgl}_n(\C)) \times
H_{3-i}^{\Ad \circ \rho}(Y, \partial Y; \mathfrak{pgl}_n(\C)) \to \C.
\end{equation}

Let $M$ be a compact connected orientable $3$-manifold whose boundary consists
of $m$ tori $T_i$ and let $\gamma_i \subset T_i$ be a simple closed curve for
each $i$.
For a representation $\rho \colon \pi_1 M \to PGL_n(\C)$ a homomorphism
$\mathfrak{pgl}_n(\C)^{\pi_1 T_i} \to H_1^{\Ad \circ \rho}(M; \mathfrak{pgl}_n(\C))$,
where
$\mathfrak{pgl}_n(\C)^{\pi_1 T_i} := \{ v \in \mathfrak{pgl}_n(\C) ~;~ \Ad \circ \rho(\pi_1 T_i) v = v \}$,
is defined to
map $v$ to $[\tilde{\gamma}_i \otimes v]$ for
$v \in \mathfrak{pgl}_n(\C)^{\pi_1 T_i}$, where $\tilde{\gamma}_i$ is a lift of
$\gamma_i$ in $\widetilde{M}$.
Similarly, a homomorphism
$\mathfrak{pgl}_n(\C)^{\pi_1 T_i} \to H_2^{\Ad \circ \rho}(M; \mathfrak{pgl}_n(\C))$
is defined to map $v$ to $[\widetilde{T}_i \otimes v]$ for
$v \in \mathfrak{pgl}_n(\C)^{\pi_1 T_i}$, where $\widetilde{T}_i$ is a lift of
$T_i$ in $\widetilde{M}$.
We denote by
$\psi_1 \colon \oplus_{i=1}^m \mathfrak{pgl}_n(\C)^{\pi_1 T_i} \to H_1^{\Ad \circ \rho}(M; \mathfrak{pgl}_n(\C))$
and
$\psi_2 \colon \oplus_{i=1}^m \mathfrak{pgl}_n(\C)^{\pi_1 T_i} \to H_2^{\Ad \circ \rho}(M; \mathfrak{pgl}_n(\C))$
the direct sums of the homomorphisms for $i$ respectively.

\begin{defn} \label{defn_regular}
A representation $\rho \colon \pi_1 M \to PGL_n(\C)$ is called
\textit{$(\gamma_1, \dots, \gamma_m)$-regular} if:
\begin{itemize}
\item[(i)] $H_0^{\Ad \circ \rho}(M; \mathfrak{pgl}_n(\C)) = 0$,
\item[(ii)] $\dim \mathfrak{pgl}_n(\C)^{\pi_1 T_i} = n-1$ for each $i$,
\item[(iii)] $\psi_1 \colon \oplus_{i=1}^m \mathfrak{pgl}_n(\C)^{\pi_1 T_i} \to
H_1^{\Ad \circ \rho}(M; \mathfrak{pgl}_n(\C))$ is surjective.
\end{itemize}
\end{defn}

\begin{rem}
The above definition is equivalent to one for representations
$\pi_1 M \to SL_2(\C)$ by Porti~\cite[D\'{e}finition 3.21]{P}.
(See also \cite[Proposition 3.22]{P}.)
\end{rem}

It is easily seen that if a representation
$\rho \colon \pi_1 M \to PGL_n(\C)$ is
$(\gamma_1, \dots, \gamma_m)$-regular, then so is a conjugation of $\rho$.

The following theorem strongly depends on the works of Menal-Ferrer and
Porti~\cite{MFP1, MFP2}.
\begin{thm} \label{thm_regular}
Suppose that $M$ is a hyperbolic $3$-manifold.
Let $\rho \colon \pi_1 M \to PGL_2(\C)$ be a holonomy representation and
$\iota_n \colon PGL_2(\C) \to PGL_n(\C)$ is the representation induced by an
irreducible representation $SL_2(\C) \to SL_n(\C)$.
Then for any $\gamma_i \subset T_i$ which is not null-homologous the
composition $\iota_n \circ \rho \colon \pi_1 M \to PGL_n(\C)$ is
$(\gamma_1, \dots, \gamma_m)$-regular.
\end{thm}
 
\begin{proof}
Since $\Ad \circ \iota_n \circ \rho$ is non-commutative,
\[ H_{\Ad \circ \iota_n \circ \rho}^0(M; \mathfrak{pgl}_n(\C)) = \mathfrak{pgl}_n(\C)^{\pi_1 M} = 0. \]
Now it follows from Poincar\'{e} duality and the duality induced by the
intersection pairing \eqref{eq_duality} that
$H_0^{\Ad \circ \iota_n \circ \rho}(M; \mathfrak{pgl}_n(\C)) = 0$, which proves
the condition (i).

Since $\rho|_{\pi_1 T_i} \colon \pi_1 T_i \to $ is a parabolic representation
for each $i$, it follows from \cite[Lemma 2.1]{MFP1} that
$\dim \mathfrak{pgl}_n(\C)^{\pi_1 T_i} = n-1$ for each $i$, which proves the
condition (ii).

We denote by $X_{M, n}$ and $X_{\gamma_i, n}$ the $PGL_n(\C)$-character varieties of the fundamental groups of $M$ and $\gamma_i$ respectively.
It follows from \cite[Theorem 1.1]{MFP1} that regular functions
$X_{M, n} \to \C$ induced by symmetric polynomials of eigenvalues for
$\iota_n \circ \rho(\gamma_i)$ for all $i$ except for the determinant give
biholomorphic local coordinates of $X_{M, n}$ as a $m(n-1)$-dimensional complex
manifold.
It is easy to check that $X_{\gamma_i, n}$ has a similar biholomorphic local
coordinates as a $(n-1)$-dimensional complex manifold.
Hence the homomorphism
$T_{\chi_{\iota_n \circ \rho}} X_{M, n} \to \oplus_i T_{\chi_{\iota_n \circ \rho}} X_{\gamma_i, n}$
is an isomorphism, which implies that so is the homomorphism
$H_{\Ad \circ \iota_n \circ \rho}^1(M; \mathfrak{pgl}_n(\C)) \to \oplus_i H_{\Ad \circ \iota_n \circ \rho}^1(\gamma_i; \mathfrak{pgl}_n(\C))$
under the identifications
$T_{\chi_{\iota_n \circ \rho}} X_{M, n} = H_{\Ad \circ \iota_n \circ \rho}^1(M; \mathfrak{pgl}_n(\C))$
and
$T_{\chi_{\iota_n \circ \rho}} X_{\gamma_i, n} = H_{\Ad \circ \iota_n \circ \rho}^1(\gamma_i; \mathfrak{pgl}_n(\C))$
for each $i$.
(See also \cite[Theorem 0.3]{MFP2}.)
Now it follows from Poincar\'{e} duality and the duality induced by the
intersection pairing \eqref{eq_duality} that the homomorphism
$\oplus_i H_1^{\Ad \circ \iota_n \circ \rho}(\gamma_i; \mathfrak{pgl}_n(\C)) \to H_1^{\Ad \circ \iota_n \circ \rho}(M; \mathfrak{pgl}_n(\C))$
is an isomorphism.
Since
\[ \dim \mathfrak{pgl}_n(\C)^{\pi_1 T_i} = \dim H_1^{\Ad \circ \iota_n \circ \rho}(\gamma_i; \mathfrak{pgl}_n(\C)) = n-1, \]
the homomorphism
$\mathfrak{pgl}_n(\C)^{\pi_1 T_i} \to H_1^{\Ad \circ \iota_n \circ \rho}(\gamma_i; \mathfrak{pgl}_n(\C))$
mapping $v$ to $[\tilde{\gamma}_i \otimes v]$ for
$v \in \mathfrak{pgl}_n(\C)^{\pi_1 T_i}$ is an isomorphism for each $i$.
Therefore
$\psi_1 \colon \oplus_{i=1}^m \mathfrak{pgl}_n(\C)^{\pi_1 T_i} \to H_1^{\Ad \circ \rho}(M; \mathfrak{pgl}_n(\C))$,
which is a composition of the above homomorphisms, is also an isomorphism,
which proves the condition (iii).
\end{proof}

\begin{lem} \label{lem_T-homology}
If a representation $\rho \colon \pi_1 T^2 \to PGL_n(\C)$ satisfies that
$\dim \mathfrak{pgl}_n(\C)^{\pi_1 T^2} = n-1$, then
\begin{align*}
&(i) \dim H_0^{\Ad \circ \rho}(T^2; \mathfrak{pgl}_n(\C)) =
\dim H_2^{\Ad \circ \rho}(T^2; \mathfrak{pgl}_n(\C)) = n-1, \\
&(ii) \dim H_1^{\Ad \circ \rho}(T^2; \mathfrak{pgl}_n(\C)) = 2(n-1). \\
\end{align*}
\end{lem}

\begin{proof}
Since $H_2^{\Ad \circ \rho}(T^2; \mathfrak{pgl}_n(\C))$ is isomorphic to
$\mathfrak{pgl}_n(\C)^{\pi_1 T^2}$, 
\[ \dim H_2^{\Ad \circ \rho}(T^2; \mathfrak{pgl}_n(\C)) =
\dim \mathfrak{pgl}_n(\C)^{\pi_1 T^2} = n-1. \]
It follows from the duality induced by the intersection pairing
\eqref{eq_duality} that
\[ \dim H_0^{\Ad \circ \rho}(T^2; \mathfrak{pgl}_n(\C)) =
\dim H_2^{\Ad \circ \rho}(T^2; \mathfrak{pgl}_n(\C)) = n-1. \]
Since
\[ \sum_{i=0}^2 (-1)^i \dim H_i^{\Ad \circ \rho}(T^2; \mathfrak{pgl}_n(\C)) =
(n^2-1) \chi(M) = 0, \]
we have
\[ \dim H_1^{\Ad \circ \rho}(T^2; \mathfrak{pgl}_n(\C)) =
\dim H_0^{\Ad \circ \rho}(T^2; \mathfrak{pgl}_n(\C)) +
\dim H_2^{\Ad \circ \rho}(T^2; \mathfrak{pgl}_n(\C)) = 2(n-1). \]
\end{proof}

\begin{lem} \label{lem_M-homology}
If a representation $\rho \colon \pi_1 M \to PGL_n(\C)$ is
$(\gamma_1, \dots, \gamma_m)$-regular for $\gamma_1, \dots, \gamma_m$, then
\[ \dim H_1^{\Ad \circ \rho}(M; \mathfrak{pgl}_n(\C)) =
\dim H_2^{\Ad \circ \rho}(M; \mathfrak{pgl}_n(\C)) = m(n-1). \]
\end{lem}

\begin{proof}
Since $\psi_1 \colon \oplus_{i=1}^m \mathfrak{pgl}_n(\C)^{\Ad \circ \rho(\pi_1 T_i)}
\to H_1^{\Ad \circ \rho}(M; \mathfrak{pgl}_n(\C))$ is surjective, so is the
homomorphism $H_1^{\Ad \circ \rho}(\partial M; \mathfrak{pgl}_n(\C)) \to
H_1^{\Ad \circ \rho}(M; \mathfrak{pgl}_n(\C))$.
It follows from the duality induced by the intersection pairing
\eqref{eq_duality} that the dual homomorphism
$H_2^{\Ad \circ \rho}(M, \partial M; \mathfrak{pgl}_n(\C)) \to
H_1^{\Ad \circ \rho}(\partial M; \mathfrak{pgl}_n(\C))$ is injective.
Now the homology long exact sequence for the pair $(M, \partial M)$ gives the
exact sequence
\[ 0 \to H_2^{\Ad \circ \rho}(M, \partial M; \mathfrak{pgl}_n(\C)) \to
H_1^{\Ad \circ \rho}(\partial M; \mathfrak{pgl}_n(\C)) \to
H_1^{\Ad \circ \rho}(M; \mathfrak{pgl}_n(\C)) \to 0. \]
Hence by Lemma \ref{lem_T-homology} (ii) we obtain
\begin{align*}
\dim H_1^{\Ad \circ \rho}(M; \mathfrak{pgl}_n(\C)) &=
\frac{1}{2} \dim H_1^{\Ad \circ \rho}(\partial M; \mathfrak{pgl}_n(\C)) \\
&= \frac{1}{2} \sum_{i=1}^m \dim H_1^{\Ad \circ \rho}(T_i; \mathfrak{pgl}_n(\C)) =
m(n-1).
\end{align*}
Since
\[ \sum_{i=0}^3 (-1)^i \dim H_i^{\Ad \circ \rho}(M; \mathfrak{pgl}_n(\C)) =
(n^2-1) \chi(M) = 0, \]
we have
\[ \dim H_2^{\Ad \circ \rho}(M; \mathfrak{pgl}_n(\C)) =
\dim H_1^{\Ad \circ \rho}(T^2; \mathfrak{pgl}_n(\C)) = m(n-1). \]
\end{proof}

\begin{lem}
If a representation $\rho \colon \pi_1 M \to PGL_n(\C)$ is
$(\gamma_1, \dots, \gamma_l)$-regular for $\gamma_1, \dots, \gamma_l$, then
$\psi_1 \colon \oplus_{i=1}^m \mathfrak{pgl}_n(\C)^{\pi_1 T_i} \to
H_1^{\Ad \circ \rho}(M; \mathfrak{pgl}_n(\C))$ and
$\psi_2 \colon \oplus_{i=1}^m \mathfrak{pgl}_n(\C)^{\pi_1 T_i} \to
H_2^{\Ad \circ \rho}(M; \mathfrak{pgl}_n(\C))$ are isomorphisms.
\end{lem}

\begin{proof}
By Lemma \ref{lem_M-homology}
\[ \dim H_1^{\Ad \circ \rho}(M; \mathfrak{pgl}_n(\C)) =
\dim H_2^{\Ad \circ \rho}(M; \mathfrak{pgl}_n(\C)) =
\dim \oplus_{i=1}^m \mathfrak{pgl}_n(\C)^{\pi_1 T_i} = m(n-1). \]
Since $\psi_1 \colon \oplus_{i=1}^m \mathfrak{pgl}_n(\C)^{\pi_1 T_i} \to
H_1^{\Ad \circ \rho}(M; \mathfrak{pgl}_n(\C))$ is surjective, it is an isomorphism.
Since $H_0^{\Ad \circ \rho}(M; \mathfrak{pgl}_n(\C)) = 0$, it follows from the duality
induced by the intersection pairing \eqref{eq_duality} that
$H_3^{\Ad \circ \rho}(M, \partial M; \mathfrak{pgl}_n(\C)) = 0$.
Now the homology long exact sequence for the pair $(M, \partial M)$ implies
that the homomorphism $H_2(\partial M; \mathfrak{pgl}_n(\C)) \to
H_2^{\Ad \circ \rho}(M; \mathfrak{pgl}_n(\C))$ is injective.
Hence $\psi_2 \colon \oplus_{i=1}^m \mathfrak{pgl}_n(\C)^{\pi_1 T_i} \to
H_2^{\Ad \circ \rho}(M; \mathfrak{pgl}_n(\C))$ is also injective, and so it is an
isomorphism. 
\end{proof}

\begin{defn}
For a $(\gamma_1, \dots, \gamma_m)$-regular representation
$\rho \colon \pi_1 M \to PGL_n(\C)$ we define the
\textit{non-acyclic Reidemeister torsion}
$T_{(\gamma_1, \dots, \gamma_m), \rho}(M)$ associated to
$(\gamma_1, \dots, \gamma_m)$ and $\rho$ as follows.
We choose a basis $b_i$ of $\mathfrak{pgl}_n(\C)^{\pi_1 T_i}$ for each $i$.
Then
\[ T_{(\gamma_1, \dots, \gamma_m), \rho}(M) =
\tau_{\Ad \circ \rho}(M; h_1\cup h_2) ~\in \C^\times / \pm 1, \]
where
\begin{align*}
h_1 &:= \langle \psi_1(b_1), \dots, \psi_1(b_m) \rangle, \\
h_2 &:= \langle \psi_2(b_1), \dots, \psi_2(b_m) \rangle.
\end{align*}
\end{defn}

It can be checked as follows that $T_{(\gamma_1, \dots, \gamma_m), \rho}(M)$ does
not depend on the choice of $b_i$.
Let $b_i'$ be another basis of $\mathfrak{pgl}_n(\C)^{\pi_1 T_i}$ for each $i$, and set
\begin{align*}
h_1' &:= \langle \psi_1(b_1'), \dots, \psi_1(b_m') \rangle, \\
h_2' &:= \langle \psi_2(b_1'), \dots, \psi_2(b_m') \rangle.
\end{align*}
Then by the definition of Reidemeister torsion we have
\[ \tau_{\Ad \circ \rho}(M; h_1'\cup h_2') = \frac{[h_1'/h_1]}{[h_2' / h_2]} \tau_{\Ad \circ \rho}(M; h_1\cup h_2). \]
and an easy computation implies
\[ [h_1' / h_1] = [h_2' / h_2] = \prod_{i=1}^m [b_i' / b_i], \]
which shows the independence.

\subsection{Non-acyclic Reidemeister torsion for fibered $3$-manifolds}

We show a formula computing non-acyclic Reidemeister torsion of fibered
$3$-manifolds from the monodromy maps.
The formula generalizes a homological version of \cite[Main Theorem]{D} for
fibered knots and $2$-dimensional representations. 

\begin{thm} \label{thm_FT}
Let $\gamma_1, \dots, \gamma_m$ be the boundary components of $S$ and let
$\varphi \in \Gamma_S$.
For a $(\gamma_1, \dots, \gamma_m)$-regular representation
$\rho \colon \pi_1 M_\varphi \to PGL_n(\C)$ satisfying
$H_0^{\Ad \circ \rho}(S; \mathfrak{pgl}_n(\C)) = 0$,
\[ T_{(\gamma_1, \dots, \gamma_m), \rho}(M_\varphi) =
\lim_{t \to 1} \frac{\det(t \varphi_* - id)}{(t-1)^{m(n-1)}}, \]
where we consider
$\varphi_* \colon H_1^{\Ad \circ \rho}(S; \mathfrak{pgl}_n(\C)) \to H_1^{\Ad \circ \rho}(S; \mathfrak{pgl}_n(\C))$
in the formula.
\end{thm}

\begin{proof}
We denote by $\psi_1' \colon \mathfrak{pgl}_n(\C)^{\pi_1 T_i} \to
H_1^{\Ad \circ \rho}(S; \mathfrak{pgl}_n(\C))$ the factor of
$\psi_1 \colon \mathfrak{pgl}_n(\C)^{\pi_1 T_i} \to
H_1^{\Ad \circ \rho}(M_\varphi; \mathfrak{pgl}_n(\C))$.
It follows from the duality induced by the intersection pairing
\eqref{eq_duality} that $H_2^{\Ad \circ \rho}(S, \partial S; \mathfrak{pgl}_n(\C)) =
H_0^{\Ad \circ \rho}(S; \mathfrak{pgl}_n(\C)) = 0$.
Now the homology long exact sequence for the pair $(S, \partial S)$ implies
that the homomorphism $H_1^{\Ad \circ \rho}(\partial S; \mathfrak{pgl}_n(\C)) \to
H_1^{\Ad \circ \rho}(S; \mathfrak{pgl}_n(\C))$ is injective, and so is $\psi_1'$.
Choose a basis $b_i$ of $\mathfrak{pgl}_n(\C)^{\pi_1 T_i}$ for each $i$ and take a
basis $h = \langle \psi_1'(b_1), \dots, \psi_1'(b_m) \rangle \cup b$ of
$H_1^{\Ad \circ \rho}(S; \mathfrak{pgl}_n(\C))$, by adding subbasis $b$.

Take a representative of $\varphi$ and a triangulation of $S$ such that the
representative is simplicial, and consider the following exact sequence:
\[ 0 \to C_*^{\Ad \circ \rho}(\widetilde{S}) \otimes \mathfrak{pgl}_n(\C)
\xrightarrow{id \times 1 - \varphi_* \times 0}
C_*^{\Ad \circ \rho}(\widetilde{S} \times [0, 1]) \otimes \mathfrak{pgl}_n(\C) \to
C_*^{\Ad \circ \rho}(\widetilde{M}_\varphi) \otimes \mathfrak{pgl}_n(\C) \to 0. \]
By Lemma \ref{lem_M}
\[ \tau_\rho(S \times [0, 1]; h) =
\tau_\rho(S; h) T_{(\gamma_1, \dots, \gamma_m), \rho}(M_\varphi)
\tau(\mathcal{H}_*, d), \]
where
\begin{align*}
\mathcal{H}_* &:= (0 \to H_2^{\Ad \circ \rho}(M_\varphi) \to
H_1^{\Ad \circ \rho}(S) \xrightarrow{I - \varphi_*} H_1^{\Ad \circ \rho}(S) \to
H_1^{\Ad \circ \rho}(M_\varphi) \to 0), \\
d &:= h_1\cup h \cup h \cup h_2.
\end{align*}
Since $\tau_\rho(S \times [0, 1]; h) = \tau_\rho(S; h)$, we have
\[ T_{(\gamma_1, \dots, \gamma_m), \rho}(M_\varphi) =
\tau(\mathcal{H}_*, d)^{-1}. \]

Considering the following commutative diagram of exact sequences
\[
\begin{CD}
0 @>>> H_2^{\Ad \circ \rho}(\partial M_\varphi) @>>>
\oplus_i H_1^{\Ad \circ \rho}(\gamma_i) @>0>>
\oplus_i H_1^{\Ad \circ \rho}(\gamma_i) @>>>
H_1^{\Ad \circ \rho}(\partial M_\varphi) \\
@. @VVV @VVV @VVV @VVV\\
0 @>>> H_2^{\Ad \circ \rho}(M_\varphi) @>>> H_1^{\Ad \circ \rho}(S)
@> id-\varphi_* >> H_1^{\Ad \circ \rho}(S) @>>>
H_1^{\Ad \circ \rho}(M_\varphi) @>>> 0,
\end{CD}
\]
where we omit to write the coefficient $\mathfrak{pgl}_n(\C)$,
we see that the homomorphism $H_1^{\Ad \circ \rho}(S; \mathfrak{pgl}_n(\C)) \to
H_1^{\Ad \circ \rho}(M_\varphi; \mathfrak{pgl}_n(\C))$ maps
$\langle \psi_1'(b_1), \dots, \psi_1'(b_m) \rangle$ to $h_1$ and that the
homomorphism $H_2^{\Ad \circ \rho}(M_\varphi; \mathfrak{pgl}_n(\C)) \to
H_1^{\Ad \circ \rho}(S; \mathfrak{pgl}_n(\C))$ maps $h_2$ to
$\langle \psi_1'(b_1), \dots, \psi_1'(b_m) \rangle$.
Therefore
\[ \tau(\mathcal{H}_*, d)^{-1} =
\det((id - \varphi_*) \colon \coker \psi_1' \to \coker \psi_1') =
\pm \lim_{t \to 1} \frac{\det(t \varphi_* - id)}{(t-1)^{m(n-1)}}, \]
which proves the theorem.
\end{proof}

For a later use, we recall a well-known formula of
`twisted Alexander polynomials' for fibered $3$-manifolds.
See for instance \cite{M2}.
\begin{lem} \label{lem_FA}
Let $\varphi \in \Gamma_S$ and let
$\psi \colon \pi_1 M_\varphi \to \langle t \rangle$ be the homomorphism induced
by the fibration.
For a representation $\rho \colon \pi_1 M_\varphi \to GL_n(V)$ over $\F$,
\[ \tau_{\psi \otimes \rho}(M_\varphi) = \frac{\det(t \varphi_1 - id)}{\det(t \varphi_0 - id)}, \]
where $\varphi_0 \colon H_0^{\rho}(S; V) \to H_0^{\rho}(S; V)$,
$\varphi_1 \colon H_1^{\rho}(S; V) \to H_1^{\rho}(S; V)$ are the homomorphisms
induced by $\varphi$.
\end{lem}

The following is a direct corollary of Theorem \ref{thm_FT} and Lemma
\ref{lem_FA}.
\begin{cor} \label{cor_L}
Let $\gamma_1, \dots, \gamma_m$ be the boundary components of $S$, let
$\varphi \in \Gamma_S$ and let
$\psi \colon \pi_1 M_\varphi \to \langle t \rangle$ be the homomorphism induced
by the fibration.
For a $(\gamma_1, \dots, \gamma_m)$-regular representation
$\rho \colon \pi_1 M_\varphi \to PGL_n(\C)$ satisfying
$H_0^{\Ad \circ \rho}(S; \mathfrak{pgl}_n(\C)) = 0$,
\[ T_{(\gamma_1, \dots, \gamma_m), \rho}(M_\varphi) = \lim_{t \to 1} \frac{\tau_{\psi \otimes \Ad \circ \rho}(M_\varphi)}{(t-1)^{m(n-1)}}, \]
where we consider
$\varphi_* \colon H_1^{\Ad \circ \rho}(S; \mathfrak{pgl}_n(\C)) \to H_1^{\Ad \circ \rho}(S; \mathfrak{pgl}_n(\C))$
in the formula.
\end{cor}

%%%%%%% Section 4 %%%%%%%%%%%%%%%%%%%%%%%%%%%%%%%%%%%%%%%%%%%%%%%%%%%%%%%%%%%%%
\section{Main theorems}

\subsection{Proof}

In this section we show the main theorems on torsion invariants and cluster algebras for surfaces.

Recall that Fock and Goncharov constructed a regular map $\nu_T \colon \mathcal{X}_{T, n} \to \mathcal{X}_{S, n}$ for an ideal triangulation $T$ of $S$, and that $T_{\chi_\rho} \mathcal{X}_{S, n}$ is identified with a subspace of $H_{\Ad \circ \rho}^1(\pi_1 S; \mathfrak{pgl}_n(\C))$ for a representation $\rho \colon \pi_1 S \to PGL_n(\C)$.
Thus $\nu_T$ induces a map $T_{(y_1, \dots, y_l)} \mathcal{X}_{T, n} \to H_{\Ad \circ \rho}^1(\pi_1 S; \mathfrak{pgl}_n(\C))$.

\begin{lem} \label{lem_tangent}
Let $(y_1, \dots, y_l) \in \mathcal{X}_{T, n}$ and
$\rho \colon \pi_1 S \to PGL_n(\C)$ a representation such that
$\chi_{(y_1, \dots, y_l)} = \chi_\rho$.
If $H_{\Ad \circ \rho}^0(\pi_1 S; \mathfrak{pgl}_n(\C)) = 0$, then the map
$T_{(y_1, \dots, y_l)} \mathcal{X}_{T, n} \to H_{\Ad \circ \rho}^1(\pi_1 S; \mathfrak{pgl}_n(\C))$
is an isomorphism.
\end{lem}

\begin{proof}
Since $T_{\chi_\rho} X_{S, n}$ embeds in
$H_{\Ad \circ \rho}^1(\pi_1 S; \mathfrak{pgl}_n(\C))$, we have
\[ l = \dim X_{S, n} \leq T_{\chi_\rho} X_{S, n} \leq H_{\Ad \circ \rho}^1(\pi_1 S; \mathfrak{pgl}_n(\C)) = l, \]
and so the inequalities are all equalities.
Moreover, since for $\chi_{(\rho, B_1, \dots, B_m)} \in \mathcal{X}_{S, n}$,
$(d \pi)_{\chi_{(\rho, B_1, \dots, B_m)}} \colon T_{\chi_{(\rho, B_1, \dots, B_m)}} \mathcal{X}_{S, n} \to T_{\chi_\rho} X_{S, n}$
is an epimorphism, we have
\[ l = T_{\chi_\rho} X_{S, n} \leq T_{(y_1, \dots, y_l)} \mathcal{X}_{T, n} = l, \]
and the inequality is an equality, which proves the lemma.
\end{proof}

Now we prove the following main theorems:
%Note that for $\varphi \in \Gamma_S$ and a representation $\rho \colon \pi_1 M_\varphi \to PGL_n(\C)$ the character $\chi_{\rho|_{\pi_1 S}}$ is an element of the fixed point set $X_{S, n}^{\varphi^*}$.

\begin{thm} \label{thm_CA}
Let $\varphi \in \Gamma_S$.
For a representation $\rho \colon \pi_1 M_\varphi \to PGL_n(\C)$ satisfying
$H_0^{\Ad \circ \rho}(S; \mathfrak{pgl}_n(\C)) = 0$ and
$(y_1^0, \dots, y_l^0) \in \mathcal{X}_{T, n}^{\varphi^*}$, if
$\chi_{\rho|_{\pi_1 S}} = \chi_{(y_1^0, \dots, y_l^0)}$, then
\[ \tau_{\psi \otimes \Ad \circ \rho}(M_\varphi) = \left. \det \left( t \left( \frac{\partial \varphi^*(y_j)}{\partial y_i} \right) - I \right) \right|_{(y_1, \dots, y_l) = (y_1^0, \dots, y_l^0)}. \]
\end{thm}

\begin{proof}
In the following the coefficients of all the twisted homology groups and all
the twisted cohomology groups are understood to be $\mathfrak{pgl}_n(\C)$.

It follows from Lemma \ref{lem_tangent} and Corollary \ref{cor_equivariance} that
the homomorphism
$\varphi^* \colon H_{\Ad \circ \rho}^1(\pi_1 S) \to H_{\Ad \circ \rho}^1(\pi_1 S)$
is presented by the matrix
$\left( \frac{\partial \varphi_*(y_j)}{\partial y_i} \right)$.
Since $H_{\Ad \circ \rho}^1(\pi_1 S)$ is isomorphic to
$H_{\Ad \circ \rho}^1(S)$ and since $H_{\Ad \circ \rho}^1(S)$ is isomorphic by
Poincar\'{e} duality to the dual of $H_1^{\Ad \circ \rho}(S, \partial S)$, the
homomorphism
$\varphi_* \colon H_1^{\Ad \circ \rho}(S, \partial S) \to H_1^{\Ad \circ \rho}(S, \partial S)$
is presented by the transpose of
$\left( \frac{\partial \varphi_*(y_j)}{\partial y_i} \right)$.
Thus by Lemma \ref{lem_FA} we only need to show that the homomorphisms
$\varphi_* \colon H_1^{\Ad \circ \rho}(S) \to H_1^{\Ad \circ \rho}(S)$ and
$\varphi_* \colon H_1^{\Ad \circ \rho}(S, \partial S) \to H_1^{\Ad \circ \rho}(S, \partial S)$
are equivalent to each other.

It follows from the duality induced by the intersection pairing
\eqref{eq_duality} that
$H_2^{\Ad \circ \rho}(S, \partial S) = H_0^{\Ad \circ \rho}(S) = 0$.
Hence the homology long exact sequence for the pair $(M, \partial M)$ gives the
following commutative diagram of exact sequences:
\[
\begin{CD}
0 @>>> H_1^{\Ad \circ \rho}(\partial S) @>>> H_1^{\Ad \circ \rho}(S) @>>> H_1^{\Ad \circ \rho}(S, \partial S) @>>> H_0^{\Ad \circ \rho}(\partial S) @>>> 0 \\
@. @| @V \varphi_* VV @V \varphi_* VV @| \\
0 @>>> H_1^{\Ad \circ \rho}(\partial S) @>>> H_1^{\Ad \circ \rho}(S) @>>> H_1^{\Ad \circ \rho}(S, \partial S) @>>> H_0^{\Ad \circ \rho}(\partial S) @>>> 0,
\end{CD}
\]
where it follows again from the duality induced by the intersection pairing
\eqref{eq_duality} that $H_0^{\Ad \circ \rho}(\partial S)$ is isomorphic to the
dual of $H_1^{\Ad \circ \rho}(\partial S)$.
Now it is a simple matter to check that
$\varphi_* \colon H_1^{\Ad \circ \rho}(S) \to H_1^{\Ad \circ \rho}(S)$ and
$\varphi_* \colon H_1^{\Ad \circ \rho}(S, \partial S) \to H_1^{\Ad \circ \rho}(S, \partial S)$
are equivalent, which completes the proof.
\end{proof}

The proof of the following theorem is now straightforward from Corollary
\ref{cor_L} and Theorem \ref{thm_CA}.

\begin{thm} \label{thm_CT}
Let $\gamma_1, \dots, \gamma_m$ be the boundary components of $S$, and let
$\varphi \in \Gamma_S$.
For a $(\gamma_1, \dots, \gamma_m)$-regular representation
$\rho \colon \pi_1 M_\varphi \to PGL_n(\C)$ satisfying
$H_0^{\Ad \circ \rho}(S; \mathfrak{pgl}_n(\C)) = 0$ and
$(y_1^0, \dots, y_l^0) \in \mathcal{X}_{T, n}^{\varphi^*}$, if
$\chi_{\rho|_{\pi_1 S}} = \chi_{(y_1^0, \dots, y_l^0)}$, then
\[ T_{(\gamma_1, \dots, \gamma_m), \rho}(M_\varphi) = \lim_{t \to 1} \frac{ \left. \det \left( t \left( \frac{\partial \varphi^*(y_j)}{\partial y_i} \right) - I \right) \right|_{(y_1, \dots, y_l) = (y_1^0, \dots, y_l^0)}}{(t-1)^{m(n-1)}}. \]
\end{thm}

\begin{rem}
It follows from Theorems \ref{thm_holonomy} and \ref{thm_regular} that the assumptions of the above theorems are satisfied for a pseudo-Anosov $\varphi \in \Gamma_S$ and a holonomy representation of $M_\varphi$. 
\end{rem}

The advantage of our main theorems is that cluster variables naturally describe torsion invariants as functions on the moduli spaces of representations in a combinatorial way.
In fact, the rational function induced by the coefficients of the polynomial
\[ \det \left( t \left( \frac{\partial \varphi^*(y_j)}{\partial y_i} \right) - I \right) \]
or that by
\[  \lim_{t \to 1} \frac{ \det \left( t \left( \frac{\partial \varphi^*(y_j)}{\partial y_i} \right) - I \right)}{(t-1)^{m(n-1)}} \]
in the theorems can be algorithmically computed from the ideal triangulation $T$ and a sequence of flips representing $\varphi$, and now regarded as torsion functions on the moduli spaces.

\begin{rem}
In \cite{NTY} the cluster variables in $\mathcal{X}_{T, 2}$ are interpreted as the shape parameters of ideal tetrahedra of $M_\varphi$, and the volumes are also explicitly computed from the cluster variables.
This is one advantage with the cluster variables to parametrize representations.
For example, this is very useful for identifying the complete holonomy representation.
\end{rem}

The following question concerning the condition on the cluster variables that ensures $(\gamma_1, \dots, \gamma_m)$-regularity naturally arises from Theorem \ref{thm_CT}:
\begin{q}
Let $\gamma_1, \dots, \gamma_m$ be the boundary components of $S$, and let
$\varphi \in \Gamma_S$.
For a representation $\rho \colon \pi_1 M_\varphi \to PGL_n(\C)$ and
$(y_1^0, \dots, y_l^0) \in \mathcal{X}_{T, n}^{\varphi^*}$, if
$\chi_{\rho|_{\pi_1 S}} = \chi_{(y_1^0, \dots, y_l^0)}$, and if
\[ \lim_{t \to 1} \frac{ \left. \det \left( t \left( \frac{\partial \varphi^*(y_j)}{\partial y_i} \right) - I \right) \right|_{(y_1, \dots, y_l) = (y_1^0, \dots, y_l^0)}}{(t-1)^{m(n-1)}}. \]
is nonzero, then is $\rho$ a $(\gamma_1, \dots, \gamma_m)$-regular representation satisfying $H_0^{\Ad \circ \rho}(S; \mathfrak{pgl}_n(\C)) = 0$?
\end{q}

\subsection{Examples}

Finally, we demonstrate our theory for $\varphi = LR$ (the figure eight knot complement) and for $\varphi = LLR$ in the case of $n = 3$.

Let $S$ be a one-holed torus, and we identify $\overline{S}$ with $\R^2 / \Z^2$
so that the marked point corresponds to the integral points of $\R^2$.
The mapping class group $\Gamma_S = SL_2(\Z)$ is generated by the matrices
\[ L =
\begin{pmatrix}
1 & 0 \\
1 & 1
\end{pmatrix},
\quad R =
\begin{pmatrix}
1 & 1 \\
0 & 1
\end{pmatrix}.
\]
We set $\varphi = LR$, and then the mapping torus $M_\varphi$ is known to be
homeomorphic to the figure eight knot complement.
We consider an ideal triangulation $T$ defined by the lines $x = 0$, $x = y$,
$y = 0$ with respect to the standard coordinates $(x, y)$ of $\R^2$.
Then the quiver $Q_{T, 3}$ and the coordinates
$(y_1, y_2, y_3, y_4, y_5, y_6, y_7, y_8) \in \mathcal{X}_{T, 3}$ is given as
in Figure \ref{fig_torus}.

\begin{figure}[h]
\centering
\includegraphics[width=8cm, clip]{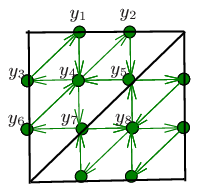}
\caption{One-holed torus $\overline{S}$ with the quiver $Q_{T, 3}$}
\label{fig_torus}
\end{figure}

A computation implies that
$L^*, R^* \colon \mathcal{X}_{T, 3} \to \mathcal{X}_{T, 3}$ are described as
follows:
\begin{align*}
L^*(y_1) = &\frac{(1+y_1) (1+y_2+y_2 y_4+y_1 y_2 y_4) y_7}{1+y_2} \\
L^*(y_2) = &\frac{(1+y_2) y_5 (1+y_1+y_1 y_8+y_1 y_2 y_8)}{1+y_1} \\
L^*(y_3) = &\frac{y_1^2 (1+y_2) y_3 y_8}{(1+y_1) (1+y_1+y_1 y_8+y_1 y_2 y_8)} \\
L^*(y_4) = &\frac{(1+y_2+y_2 y_4+y_1 y_2 y_4) y_8}{1+y_1+y_1 y_8+y_1 y_2 y_8} \\
L^*(y_5) = &\frac{1+y_1}{y_1 (1+y_2) y_8} \\
L^*(y_6) = &\frac{(1+y_1) y_2^2 y_4 y_6}{(1+y_2) (1+y_2+y_2 y_4+y_1 y_2 y_4)} \\
L^*(y_7) = &\frac{1+y_2}{(1+y_1) y_2 y_4} \\
L^*(y_8) = &\frac{y_4 (1+y_1+y_1 y_8+y_1 y_2 y_8)}{1+y_2+y_2 y_4+y_1 y_2 y_4}
\end{align*}
\begin{align*}
R^*(y_1) = &\frac{y_1 (1+y_3) (1+y_6+y_6 y_8+y_3 y_6 y_8)}{1+y_6} \\
R^*(y_2) = &\frac{y_2 (1+y_6) (1+y_3+y_3 y_4+y_3 y_4 y_6)}{1+y_3} \\
R^*(y_3) = &\frac{y_3^2 y_4 y_5 (1+y_6)}{(1+y_3) (1+y_3+y_3 y_4+y_3 y_4 y_6)} \\
R^*(y_4) = &\frac{y_4 (1+y_6+y_6 y_8+y_3 y_6 y_8)}{1+y_3+y_3 y_4+y_3 y_4 y_6} \\
R^*(y_5) = &\frac{1+y_3}{y_3 y_4 (1+y_6)} \\
R^*(y_6) = &\frac{(1+y_3) y_6^2 y_7 y_8}{(1+y_6) (1+y_6+y_6 y_8+y_3 y_6 y_8)} \\
R^*(y_7) = &\frac{1+y_6}{(1+y_3) y_6 y_8} \\
R^*(y_8) = &\frac{(1+y_3+y_3 y_4+y_3 y_4 y_6) y_8}{1+y_6+y_6 y_8+y_3 y_6 y_8}
\end{align*}

Combining them, we compute
$\varphi^* = R^* \circ L^* \colon \mathcal{X}_{T, 3} \to \mathcal{X}_{T, 3}$
as follows:
\begin{align*}
\varphi^*(y_1) = &(y_7 (1 + 2 y_1 + y_1^2 + y_1 y_8 + y_1^2 y_8 + y_1 y_2 y_8 + y_1^2 y_2 y_8 + y_1^2 y_3 y_8 + y_1^2 y_2 y_3 y_8) (1 + 2 y_2 + y_2^2 \\
&+ 2 y_2 y_4 + 2 y_1 y_2 y_4 + 2 y_2^2 y_4 + 2 y_1 y_2^2 y_4 + y_2^2 y_4^2 + 2 y_1 y_2^2 y_4^2 + y_1^2 y_2^2 y_4^2 + y_2^2 y_4 y_6 + y_1 y_2^2 y_4 y_6 \\
&+ y_2^2 y_4^2 y_6 + 2 y_1 y_2^2 y_4^2 y_6 + y_1^2 y_2^2 y_4^2 y_6 + y_1 y_2^2 y_4^2 y_6 y_8 + y_1^2 y_2^2 y_4^2 y_6 y_8 + y_1^2 y_2^2 y_3 y_4^2 y_6 y_8)) \\
&/((1 + 2 y_2 + y_2^2 + y_2 y_4 + y_1 y_2 y_4 + y_2^2 y_4 + y_1 y_2^2 y_4 + y_2^2 y_4 y_6 + y_1 y_2^2 y_4 y_6) (1 + y_1 + y_1 y_8 \\
&+ y_1 y_2 y_8)) \\
\varphi^*(y_2) = &(y_5 (1 + 2 y_2 + y_2^2 + y_2 y_4 + y_1 y_2 y_4 + y_2^2 y_4 + y_1 y_2^2 y_4 + y_2^2 y_4 y_6 + y_1 y_2^2 y_4 y_6) (1 + 2 y_1 + y_1^2 \\
&+ 2 y_1 y_8 + 2 y_1^2 y_8 + 2 y_1 y_2 y_8 + 2 y_1^2 y_2 y_8 + y_1^2 y_3 y_8 + y_1^2 y_2 y_3 y_8 + y_1^2 y_8^2 + 2 y_1^2 y_2 y_8^2 + y_1^2 y_2^2 y_8^2 \\
&+ y_1^2 y_3 y_8^2 + 2 y_1^2 y_2 y_3 y_8^2 + y_1^2 y_2^2 y_3 y_8^2 + y_1^2 y_2 y_3 y_4 y_8^2 + y_1^2 y_2^2 y_3 y_4 y_8^2 + y_1^2 y_2^2 y_3 y_4 y_6 y_8^2)) \\
&/((1 + y_2 + y_2 y_4 + y_1 y_2 y_4) (1 + 2 y_1 + y_1^2 + y_1 y_8 + y_1^2 y_8 + y_1 y_2 y_8 + y_1^2 y_2 y_8 + y_1^2 y_3 y_8 \\
&+ y_1^2 y_2 y_3 y_8)) \\
\varphi^*(y_3) = &(y_1^3 y_3^2 (1 + 2 y_2 + y_2^2 + y_2 y_4 + y_1 y_2 y_4 + y_2^2 y_4 + y_1 y_2^2 y_4 + y_2^2 y_4 y_6 + y_1 y_2^2 y_4 y_6) y_8^2) \\
&/((1 + 2 y_1 + y_1^2 + y_1 y_8 + y_1^2 y_8 + y_1 y_2 y_8 + y_1^2 y_2 y_8 + y_1^2 y_3 y_8 + y_1^2 y_2 y_3 y_8) (1 + 2 y_1 + y_1^2 \\
&+ 2 y_1 y_8 + 2 y_1^2 y_8 + 2 y_1 y_2 y_8 + 2 y_1^2 y_2 y_8 + y_1^2 y_3 y_8 + y_1^2 y_2 y_3 y_8 + y_1^2 y_8^2 + 2 y_1^2 y_2 y_8^2 + y_1^2 y_2^2 y_8^2 \\
&+ y_1^2 y_3 y_8^2 + 2 y_1^2 y_2 y_3 y_8^2 + y_1^2 y_2^2 y_3 y_8^2 + y_1^2 y_2 y_3 y_4 y_8^2 + y_1^2 y_2^2 y_3 y_4 y_8^2 + y_1^2 y_2^2 y_3 y_4 y_6 y_8^2))
\end{align*}

\begin{align*}
\varphi^*(y_4) = &(y_8 (1 + y_1 + y_1 y_8 + y_1 y_2 y_8) (1 + 2 y_2 + y_2^2 + 2 y_2 y_4 + 2 y_1 y_2 y_4 + 2 y_2^2 y_4 + 2 y_1 y_2^2 y_4 + y_2^2 y_4^2 \\
&+ 2 y_1 y_2^2 y_4^2 + y_1^2 y_2^2 y_4^2 + y_2^2 y_4 y_6 + y_1 y_2^2 y_4 y_6 + y_2^2 y_4^2 y_6 + 2 y_1 y_2^2 y_4^2 y_6 + y_1^2 y_2^2 y_4^2 y_6 + y_1 y_2^2 y_4^2 y_6 y_8 \\
&+ y_1^2 y_2^2 y_4^2 y_6 y_8 + y_1^2 y_2^2 y_3 y_4^2 y_6 y_8)) \\
&/((1 + y_2 + y_2 y_4 + y_1 y_2 y_4) (1 + 2 y_1 + y_1^2 + 2 y_1 y_8 + 2 y_1^2 y_8 + 2 y_1 y_2 y_8 + 2 y_1^2 y_2 y_8 + y_1^2 y_3 y_8 \\
&+ y_1^2 y_2 y_3 y_8 + y_1^2 y_8^2 + 2 y_1^2 y_2 y_8^2 + y_1^2 y_2^2 y_8^2 + y_1^2 y_3 y_8^2 + 2 y_1^2 y_2 y_3 y_8^2 + y_1^2 y_2^2 y_3 y_8^2 + y_1^2 y_2 y_3 y_4 y_8^2 \\
&+ y_1^2 y_2^2 y_3 y_4 y_8^2 + y_1^2 y_2^2 y_3 y_4 y_6 y_8^2)) \\
\varphi^*(y_5) = &((1 + y_1 + y_1 y_8 + y_1 y_2 y_8) (1 + 2 y_1 + y_1^2 + y_1 y_8 + y_1^2 y_8 + y_1 y_2 y_8 + y_1^2 y_2 y_8 + y_1^2 y_3 y_8 \\
&+ y_1^2 y_2 y_3 y_8)) \\
&/(y_1^2 y_3 (1 + 2 y_2 + y_2^2 + y_2 y_4 + y_1 y_2 y_4 + y_2^2 y_4 + y_1 y_2^2 y_4 + y_2^2 y_4 y_6 + y_1 y_2^2 y_4 y_6) y_8^2) \\
\varphi^*(y_6) = &(y_2^3 y_4^2 y_6^2 (1 + 2 y_1 + y_1^2 + y_1 y_8 + y_1^2 y_8 + y_1 y_2 y_8 + y_1^2 y_2 y_8 + y_1^2 y_3 y_8 + y_1^2 y_2 y_3 y_8)) \\
&/((1 + 2 y_2 + y_2^2 + y_2 y_4 + y_1 y_2 y_4 + y_2^2 y_4 + y_1 y_2^2 y_4 + y_2^2 y_4 y_6 + y_1 y_2^2 y_4 y_6) (1 + 2 y_2 + y_2^2 \\
&+ 2 y_2 y_4 + 2 y_1 y_2 y_4 + 2 y_2^2 y_4 + 2 y_1 y_2^2 y_4 + y_2^2 y_4^2 + 2 y_1 y_2^2 y_4^2 + y_1^2 y_2^2 y_4^2 + y_2^2 y_4 y_6 + y_1 y_2^2 y_4 y_6 \\
&+ y_2^2 y_4^2 y_6 + 2 y_1 y_2^2 y_4^2 y_6 + y_1^2 y_2^2 y_4^2 y_6 + y_1 y_2^2 y_4^2 y_6 y_8 + y_1^2 y_2^2 y_4^2 y_6 y_8 + y_1^2 y_2^2 y_3 y_4^2 y_6 y_8)) \\
\varphi^*(y_7) = &((1 + y_2 + y_2 y_4 + y_1 y_2 y_4) (1 + 2 y_2 + y_2^2 + y_2 y_4 + y_1 y_2 y_4 + y_2^2 y_4 + y_1 y_2^2 y_4 + y_2^2 y_4 y_6 \\
&+ y_1 y_2^2 y_4 y_6)) \\
&/(y_2^2 y_4^2 y_6 (1 + 2 y_1 + y_1^2 + y_1 y_8 + y_1^2 y_8 + y_1 y_2 y_8 + y_1^2 y_2 y_8 + y_1^2 y_3 y_8 + y_1^2 y_2 y_3 y_8)) \\
\varphi^*(y_8) = &(y_4 (1 + y_2 + y_2 y_4 + y_1 y_2 y_4) (1 + 2 y_1 + y_1^2 + 2 y_1 y_8 + 2 y_1^2 y_8 + 2 y_1 y_2 y_8 + 2 y_1^2 y_2 y_8 + y_1^2 y_3 y_8 \\
&+ y_1^2 y_2 y_3 y_8 + y_1^2 y_8^2 + 2 y_1^2 y_2 y_8^2 + y_1^2 y_2^2 y_8^2 + y_1^2 y_3 y_8^2 + 2 y_1^2 y_2 y_3 y_8^2 + y_1^2 y_2^2 y_3 y_8^2 + y_1^2 y_2 y_3 y_4 y_8^2 \\
&+ y_1^2 y_2^2 y_3 y_4 y_8^2 + y_1^2 y_2^2 y_3 y_4 y_6 y_8^2)) \\
&/((1 + y_1 + y_1 y_8 + y_1 y_2 y_8) (1 + 2 y_2 + y_2^2 + 2 y_2 y_4 + 2 y_1 y_2 y_4 + 2 y_2^2 y_4 + 2 y_1 y_2^2 y_4 + y_2^2 y_4^2 \\
&+ 2 y_1 y_2^2 y_4^2 + y_1^2 y_2^2 y_4^2 + y_2^2 y_4 y_6 + y_1 y_2^2 y_4 y_6 + y_2^2 y_4^2 y_6 + 2 y_1 y_2^2 y_4^2 y_6 + y_1^2 y_2^2 y_4^2 y_6 + y_1 y_2^2 y_4^2 y_6 y_8 \\
&+ y_1^2 y_2^2 y_4^2 y_6 y_8 + y_1^2 y_2^2 y_3 y_4^2 y_6 y_8))
\end{align*}

The space of solutions of the equations $\varphi^*(y_i) = y_i$ for all $i$ parametrizes $\mathcal{X}_{S, 3}^{\varphi^*}$ and $X_{S, 3}^{\varphi^*}$.
Here we emphasize that the parametrization makes sense by using the labeling change $\sigma$ and by Proposition \ref{prop_C} proved in the paper.

A solution is given by
\begin{align*}
y_1^0 = y_2^0 &= \frac{-1-\sqrt{-3}}{2} \\
y_3^0 = y_4^0 = y_6^0 = y_8^0 &= 1 \\
y_5^0 = y_7^0 &= \frac{-1+\sqrt{-3}}{2}.
\end{align*}
This solution can be found, for example, by using the arguments in the proofs
of Theorem \ref{thm_holonomy} and Corollary \ref{cor_solution}.
First we find an element
$\left( 1, \frac{-1+\sqrt{-3}}{2}, \frac{-1-\sqrt{-3}}{2} \right) \in \mathcal{X}_{T, 2}^{\varphi^*}$
corresponding to the character of a holonomy representation of the hyperbolic
manifold $M_\varphi$ as in \cite[Section 5.1]{NTY}.
Then the above element of $\mathcal{X}_{T, 3}^{\varphi^*}$ is the image of the
map $\mathcal{X}_{T, 2} \to \mathcal{X}_{T, 3}$ in the proof of Corollary
\ref{cor_solution}.
In fact, it corresponds to the character of the composition of a holonomy
representation and the homomorphism $PGL_2(\C) \to PGL_3(\C)$ induced by an
irreducible representation $PGL_2(\C) \to SL_3(\C)$. 

By Theorem \ref{thm_CA} we obtain the twisted Alexander polynomial associated
to the solution as:
\[ \left. \det \left( t \left( \frac{\partial \varphi^*(y_j)}{\partial y_i} \right) - I \right) \right|_{(y_1, \dots, y_l) = (y_1^0, \dots, y_l^0)} = (t-1)^2 (t^2-5t+1) (t^4-9t^3+44t-9t+1). \]
By Theorem \ref{thm_CT} we also obtain the non-acyclic torsion associated to
the solution as:
\begin{align*}
\lim_{t \to 1} \frac{ \left. \det \left( t \left( \frac{\partial \varphi^*(y_j)}{\partial y_i} \right) - I \right) \right|_{(y_1, \dots, y_l) = (y_1^0, \dots, y_l^0)}}{(t-1)^2} &= \lim_{t \to 1} (t^2-5t+1) (t^4-9t^3+44t-9t+1) \\
&= -84.
\end{align*}

Next, we set $\varphi' = LLR$.
Similarly, we can first compute $\varphi'^* = R^* \circ L^* \circ L^* \colon \mathcal{X}_{T, 3} \to \mathcal{X}_{T, 3}$ and the equations $\varphi'^*(y_i) = y_i$ for all $i$ defining $\mathcal{X}_{T, 3}^{\varphi'^*}$.
Then the following solution of the equations corresponding to the character of a holonomy representation of $M_{\varphi'^*}$ is found as follows:
\begin{align*}
y_1^0 = y_2^0 &= \frac{-3-\sqrt{-7}}{2} \\
y_3^0 = y_6^0 &= \frac{5+\sqrt{-7}}{8} \\
y_4^0 = y_8^0 &= 1 \\
y_5^0 = y_7^0 &= \frac{-1+\sqrt{-7}}{4}.
\end{align*}
Again by Theorems \ref{thm_CA} and \ref{thm_CT}, we obtain the twisted Alexander polynomial and the non-acyclic torsion associated to the solution as:
\begin{align*}
\left. \det \left( t \left( \frac{\partial \varphi'^*(y_j)}{\partial y_i} \right) - I \right) \right|_{(y_1, \dots, y_l) = (y_1^0, \dots, y_l^0)} = &(t-1)^2 (t^6+8 i \sqrt{7} t^5-22 t^5-80 i \sqrt{7} t^4+227 t^4 +208 i \sqrt{7} t^3 \\
&-1420 t^3-80 i \sqrt{7} t^2+227 t^2+8 i \sqrt{7} t-22 t+1), \\
\lim_{t \to 1} \frac{ \left. \det \left( t \left( \frac{\partial \varphi'^*(y_j)}{\partial y_i} \right) - I \right) \right|_{(y_1, \dots, y_l) = (y_1^0, \dots, y_l^0)}}{(t-1)^2} = &\lim_{t \to 1} (t^6+8 i \sqrt{7} t^5-22 t^5-80 i \sqrt{7} t^4+227 t^4+208 i \sqrt{7} t^3 \\
&-1420 t^3-80 i \sqrt{7} t^2+227 t^2+8 i \sqrt{7} t-22 t+1) \\
= &-1008 + 64 \sqrt{-7}.
\end{align*}

In the above computations on torsion invariants we specify solutions as $(y_1, \dots, y_l) = (y_1^0, \dots, y_l^0)$ for simplicity of the expressions, but note that without any specification of solutions our formulas give the torsion functions with coefficients in $(y_1, \dots, y_8) \in \mathcal{X}_{T, 3}$.

%%%%%%% References %%%%%%%%%%%%%%%%%%%%%%%%%%%%%%%%%%%%%%%%%%%%%%%%%%%%%%%%%%%%

\end{document}